\documentclass[12pt]{article}

\usepackage{amsmath,amssymb,amsthm}
\usepackage{colortbl}
\usepackage{url}

\usepackage{graphicx, fullpage}

\usepackage{nicefrac}

\usepackage[margin=1in]{geometry}

\usepackage[footnotesize]{caption}

\newtheorem{theorem}{Theorem}
\newtheorem{lemma}[theorem]{Lemma}
\newtheorem{corollary}[theorem]{Corollary}

\DeclareMathOperator{\tr}{\mathrm{tr}}
\DeclareMathOperator{\SUM}{\textsc{sum}}
\DeclareMathOperator{\HP}{HP}
\DeclareMathOperator{\HC}{HC}
\DeclareMathOperator{\SP}{SP}
\DeclareMathOperator{\SC}{SC}

\title{Making Walks Count: From Silent Circles to Hamiltonian Cycles}
\author{Max A. Alekseyev and G\'erard P. Michon}
\date{}

\begin{document}
\maketitle

Leonhard Euler (1707--1783) famously invented graph theory in 1735, by solving a puzzle of interest to the inhabitants of 
K\"onigsberg.  The city comprised three distinct land masses, connected by seven bridges. 
The residents sought a walk through the city that crossed each bridge exactly once, but were consistently unable to find one.  
Euler reduced the problem to its bare bones by representing each land mass as a node and each bridge as an edge connecting two nodes.
He then showed that such a puzzle would have a solution if and only if every node 
was at the origin of an even number of edges, 
with at most two exceptions---which could only be at the start or the
end of the journey. Since this was not the case in K\"onigsberg, the puzzle had no solution.
 The sort of diagram Euler employed, in which the nodes were represented by dots and the edges by line segments connecting the dots,
is today referred to as a \textit{graph}.  Sometimes it is convenient to use arrows instead of line segments, to imply that the connection goes in
only one direction.  
The resulting construct is now referred to as a \textit{directed graph}, or \textit{digraph} for short.

Except for tiny examples like the one inspired by K\"onigsberg, a sketch on paper is rarely an adequate description of a graph.  
One convenient representation of a digraph 
is given by its \emph{adjacency matrix} $A$, where the element $A_{i,j}$ is the number of edges going from
node $i$ to node $j$ (in a \textit{simple graph}, that number is either $0$ or $1$).
An undirected graph, like the K\"onigsberg graph, can be viewed as a digraph with a symmetric
adjacency matrix (as every undirected edge between two nodes corresponds to
a pair of directed edges going back and forth between the nodes).

A fruitful bonus of using adjacency matrices to represent graphs is that the ordinary multiplication of such matrices is surprisingly meaningful:
the $n$-th power of the adjacency matrix 
describes \emph{walks} along $n$ successive edges (not necessarily distinct) in the graph. 
This observation leads to a method called the \emph{transfer-matrix method} (e.g., see Stanley \cite[Section 4.7]{stanley}) 
that employs linear algebra techniques to enumerate walks very efficiently.
We shall perform a few spectacular enumerations using this method.

The element $A_{i,j}$ of the adjacency matrix can be viewed as the number of walks of length $1$ from node $i$ to node $j$. 
What is the number of such walks of length 2? Well, it is clearly the number of ways to go from $i$ to some node $k$ along one edge and then from that node $k$ to node $j$ along
a second edge. This amounts to the sum of the products $A_{i,k}\cdot A_{k,j}$ over all $k$, which is immediately recognized as a matrix element of the square of $A$, namely $(A^2)_{i,j}$.
More generally, the above is the pattern for a proof by induction on $n$ of the following theorem.

\begin{theorem}[{\cite[Theorem 4.7.1]{stanley}}]\label{Th1}
The number $(A^n)_{i,j}$ equals the number of walks of length $n$ going from node $i$ to node $j$
in the digraph with the adjacency matrix $A$.
\end{theorem}

A walk is called \emph{closed} if it starts and ends at the same node.
Theorem~\ref{Th1} immediately implies the following statement for the number of closed walks:

\begin{corollary}\label{Th2}
In a digraph with the adjacency matrix $A$, the number of closed walks of length $n$ equals $\tr(A^n)$, the trace of $A^n$.
\end{corollary}

It is often convenient to represent a sequence of numbers $a_0,a_1,a_2,\dots$ in the form of a \emph{generating function} $f(z)$ (of indeterminate $z$)
such that the coefficient of $z^n$ in $f(z)$ equals $a_n$ for all integers $n\geq 0$ (e.g., see \cite{wilf} for a nice introduction to generating functions).
In other words, $f(z)=a_0+a_1\cdot z+a_2\cdot z^2+\cdots$. The generating function for the number of closed walks has a neat algebraic expression:

\begin{theorem}[{\cite[Corollary 4.7.3]{stanley}}]\label{th:trgen}
For any $m\times m$ matrix $A$,
$$\sum_{n=0}^{\infty} \tr(A^{n})\cdot z^n = m-\frac{zF'(z)}{F(z)},$$
where $F(z) = \det(I_m - z\cdot A)$, and $I_m$ is the $m\times m$ identity matrix.
\end{theorem}

We will show how to put these nice results to good use by reducing 
some enumeration problems
to the counting of walks or closed walks in certain digraphs.

\section{Silent Circles}

One of our motivations for the present work was the elegant solution to
a problem originally posed by Philip Brocoum, who described the
following game as a preliminary event in a drama class he once attended at MIT.
The game was played repeatedly by all the students until silence was achieved.\footnote{Presumably, the teacher would participate only if the number of students was odd.}

An even number, $2n$, of people stand in a circle with their heads lowered.
On cue, everyone looks up and stares either at one of their two immediate neighbors
(left or right) or at the person diametrically opposed. If two people make eye
contact, both will scream! What is the probability that everyone will be silent?
For $n>1$,\footnote{The case $n=1$ is special, since the two immediate neighbors and the diametrically opposite person all coincide.}
since each person has $3$ choices, there are $3^{2n}$ possible configurations (which
are assumed to be equiprobable).
The problem then becomes just to count the number of \emph{silent configurations}.

Let us first do so in the slightly easier case of an \emph{$n$-prism} of people (we will return to the original problem later). 
This is a fancy way to say that the people are now arranged in two concentric circles each with $n$ people, where every person faces a
\emph{partner} on the other circle and is allowed to look either at that partner or at one of two neighbors on the same circle.

The key idea is to notice that the silent configurations are in one-to-one
correspondence with the closed walks of length $n$ in a certain digraph on $8$ nodes.
Indeed, there are $3^2-1=8$ different ways for the two partners in a pair to not make eye contact with each other. 
We call each such way a \emph{gaze} and denote it with a pair of arrows, one over another, indicating sight directions of the partners. 
Here the top arrow represents the person on the outer circle, while the bottom arrow represents the person on the inner circle.
An arrow pointing right indicates a person looking in the \emph{clockwise} direction (i.e., at the left neighbor on the outer circle or at the right neighbor on the inner circle).
Similarly, an arrow pointing left indicates a person looking in the \emph{counterclockwise} direction. 
Arrows pointing up or down indicate a person looking at the partner. We now build the \emph{gaze digraph}, 
whose nodes are the different gazes. There is an edge going from node $i$ to node $j$ if and only if gaze $j$ can be clockwise next to gaze $i$ in a silent configuration.

The gaze digraph and its adjacency matrix $A$ are shown in Figure~\ref{fig:gaze}. 
For example, gaze
$\left[{\rightarrow\atop\rightarrow}\right]$ denotes a pair of partners both looking at their clockwise neighbors. 
In a silent configuration, this pair can be clockwise followed by any pair, in which neither of the partners looks at the counterclockwise neighbor. 
That is, in the gaze graph, directed edges from node $\left[{\rightarrow\atop\rightarrow}\right]$ go to nodes 
$\left[{\rightarrow\atop\smash{\uparrow}\lefteqn{\phantom{\rightarrow}}}\right]$, $\left[{\downarrow\atop\rightarrow}\right]$, $\left[{\rightarrow\atop\rightarrow}\right]$ 
(in the last case, the edge forms a self-loop).

\begin{figure}[!t]
\begin{tabular}{cc}
\begin{tabular}{c}
\includegraphics[width=0.45\textwidth]{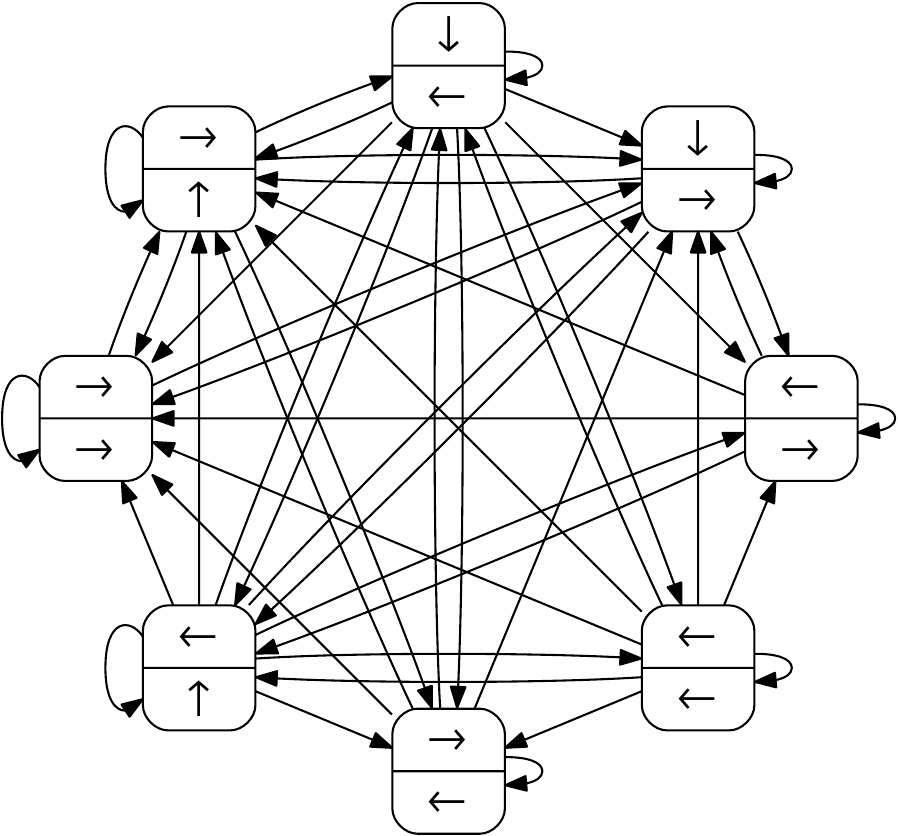}
\end{tabular}
&
\begin{tabular}{c||c|c|c|c|c|c|c|c}
 & ${\rightarrow\atop\rightarrow}$ & ${\leftarrow\atop\leftarrow}$ & ${\leftarrow\atop\smash{\uparrow}\lefteqn{\phantom{\rightarrow}}}$ & ${\downarrow\atop\rightarrow}$
 & ${\leftarrow\atop\rightarrow}$ & ${\rightarrow\atop\leftarrow}$ & ${\rightarrow\atop\smash{\uparrow}\lefteqn{\phantom{\rightarrow}}}$ & ${\downarrow\atop\leftarrow}$\\
\hline\hline
${\rightarrow\atop\rightarrow}$ &	\cellcolor{green}1 &	 0 &	0 &	1 &	0 &	0 &	1 &	0\\
\hline
${\leftarrow\atop\leftarrow}$  &	 1 &	\cellcolor{green}1 &	1 &	1 &	1 &	1 &	1 &	1\\
\hline
${\leftarrow\atop\smash{\uparrow}\lefteqn{\phantom{\rightarrow}}}$    &	1 & 	1 &	1 &	 1 &	1 &	1 &	1 &	\cellcolor{green}1\\
\hline
${\downarrow\atop\rightarrow}$ & 	1 &	0 &	 1 &	1 &	1 &	0 &	\cellcolor{green}1 &	0\\
\hline
${\leftarrow\atop\rightarrow}$ & 	1 &	0 &	1 &	1 &	1 &	\cellcolor{green}0 &	1 &	0\\
\hline
${\rightarrow\atop\leftarrow}$ & 	1 &	0 &	0 &	1 &	\cellcolor{green}0 &	1 &	1 &	1\\
\hline
${\rightarrow\atop\smash{\uparrow}\lefteqn{\phantom{\rightarrow}}}$ & 	1 &	0 &	0 &	\cellcolor{green}1 &	0 &	1 &	1 &	1\\
\hline
${\downarrow\atop\leftarrow}$ & 	1 &	1 &	\cellcolor{green}1 &	1 &	1 &	1 &	1 &	1
\end{tabular}
\end{tabular}
\caption{The gaze digraph and its adjacency matrix $A$.}
\label{fig:gaze}
\end{figure}

Let $t_n$ be the number of silent configurations of the $n$-gonal prism.
By Corollary~\ref{Th2}, we have $t_n=\tr(A^n)$. 
Theorem~\ref{th:trgen} further implies (by direct calculation) that

$$\sum_{n=0}^\infty t_n\cdot z^n = \frac{8 - 56z + 96z^2 - 50z^3 + 4z^4}{1-8z+16z^2-10z^3+z^4}.$$

From this generating function, we can easily derive a recurrence relation for $t_n$. Multiply the generating function by $1-8z+16z^2-10z^3+z^4$ to get

$$(1-8z+16z^2-10z^3+z^4)\cdot \sum_{n=0}^\infty t_n\cdot z^n = 8 - 56z + 96z^2 - 50z^3 + 4z^4.$$

For $n\geq 5$, the equality of the coefficients of $z^n$ on the left-hand and right-hand sides gives

$$t_n - 8t_{n-1} + 16t_{n-2} - 10t_{n-3} + t_{n-4} = 0.$$

The values of $t_n$ form the sequence \texttt{A141384} in the OEIS~\cite{oeis}.

Returning to the original problem, let $s_n$ be the number of silent configurations of a circle with $2n$ people.
In this problem, each gaze is formed by two diametrically opposite people on the circle.
For $n>1$, a silent configuration is therefore defined by a walk of length $n$, 
where the starting and ending nodes represent the same pair of people in a different order.
In other words, the starting and ending gazes must be obtained from each other by a vertical flip.
The entries of the adjacency matrix in Figure~\ref{fig:gaze} corresponding to such gaze flips are colored green. 
By Theorem~\ref{Th1}, the number $s_n$ equals the sum of the elements in these entries in the matrix $A^n$.
Since the minimal polynomial of $A$ is
$$x^5 - 8 x^4 + 16 x^3 - 10 x^2 + x,$$
the sequence $s_n$ (sequence \texttt{A141221} in the OEIS~\cite{oeis}) satisfies the recurrence relation:
$$s_{n} = 8 s_{n-1} - 16 s_{n-2} + 10 s_{n-3} - s_{n-4},\qquad n\geq 6,$$
which matches that for $t_n$. 
Taking into account the initial values of $s_n$ for $n=2,3,4,5$, we further deduce the generating function
$$\sum_{n=2}^\infty s_n\cdot z^n = \frac{30z^2-84z^3+58z^4-6z^5}{1-8z+16z^2-10z^3+z^4}.$$

We give initial numerical values of the sequences $t_n$ and $s_n$,
along with the corresponding probabilities of silent configurations, in the table below. 
Quite remarkably, we have $t_n = s_n + 2$ for all $n>1$.
It further follows that both probabilities $\nicefrac{t_n}{3^{2n}}$ and $\nicefrac{s_n}{3^{2n}}$ grow as 
$(\nicefrac{\alpha}{9})^n \approx 0.5948729^n$, where 
$$
\alpha = \frac{1}{3} \left(7 + 2\cdot \sqrt{22}\cdot \cos\left(\frac{\arctan(\nicefrac{\sqrt{5319}}{73})}{3}\right) \right) \approx 5.353856
$$ 
is the largest zero of the minimal polynomial of $A$.

\

\begin{center}
\begin{footnotesize}
\begin{tabular}{|c||c|c|c|c|c|c|c|c|c|}
\hline
$n$ & 2 & 3 & 4 & 5 & 6 & 7 & 8 & 9 & 10 \\
\hline\hline
$t_n$ & 32 & 158 & 828 & 4408 & 23564 & 126106 & 675076 & 3614144 & 19349432 \\
\hline
$\nicefrac{t_n}{3^{2n}}$ & 0.395 & 0.217 & 0.126 & 0.075 & 0.044 & 0.026 & 0.016 & 0.009 & 0.006 \\
\hline\hline
$s_n$ & 30 & 156 & 826 & 4406 & 23562 & 126104 & 675074 & 3614142 & 19349430 \\
\hline
$\nicefrac{s_n}{3^{2n}}$ & 0.370 & 0.214 & 0.126 & 0.075 & 0.044 & 0.026 & 0.016 & 0.009 & 0.006 \\
\hline
\end{tabular}
\end{footnotesize}
\end{center}

\section{Hamiltonian Cycles in Antiprism Graphs}

An \emph{antiprism graph} represents the skeleton of an antiprism. The $n$-antiprism graph (defined for $n\geq 3$) has $2n$ nodes and $4n$ edges. 
Its nodes can be placed around a circle and enumerated with the numbers from $0$ to $2n-1$ 
such that each node $i$ ($i=0,1,\dots,2n-1$) is connected to nodes\footnote{From now on, we assume that arithmetic operations on node labels are done modulo $2n$.} $i\pm 1$ and $i\pm 2$
(an example for $n=5$ is shown in Figure~\ref{fig:anti}a). 
These graphs represent a special case of the more general \textit{circulant graphs} and are denoted $C_{2n}^{1,2}$ (here the subscript specifies the number of nodes, while the 
superscript describes the pattern for edges).

\begin{figure}[!t]
\begin{tabular}{cc}
\includegraphics[width=0.45\textwidth]{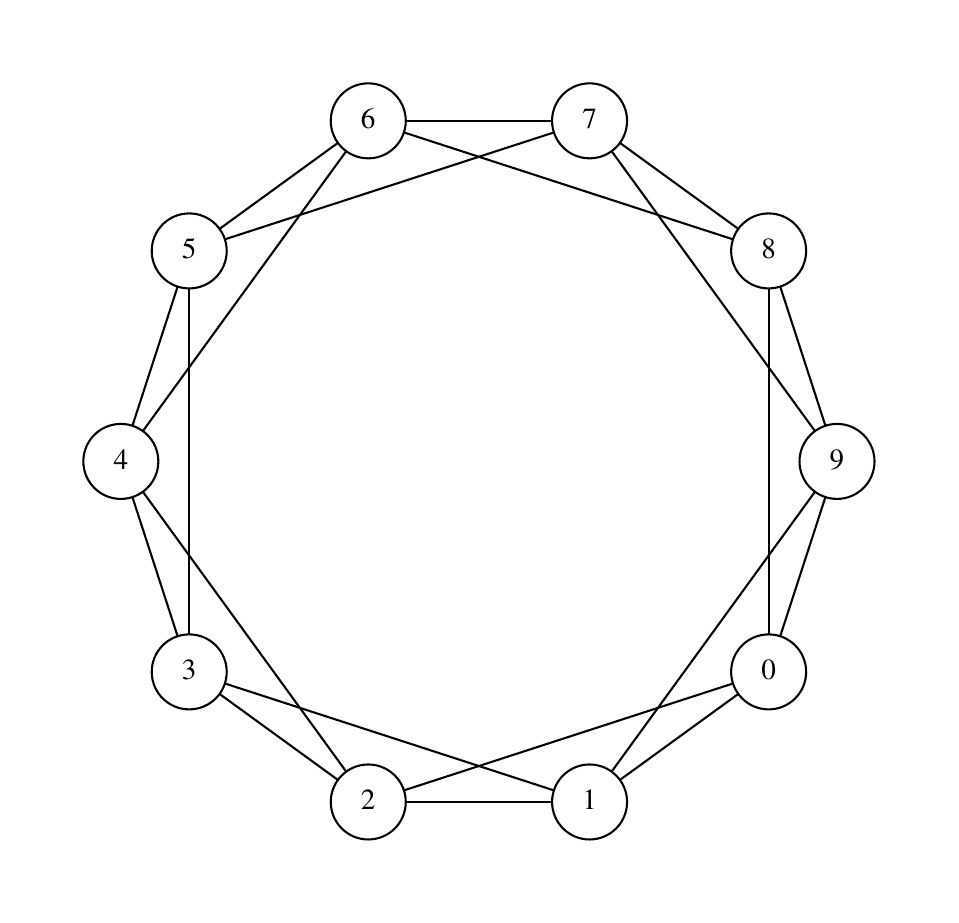}
&
\includegraphics[width=0.45\textwidth]{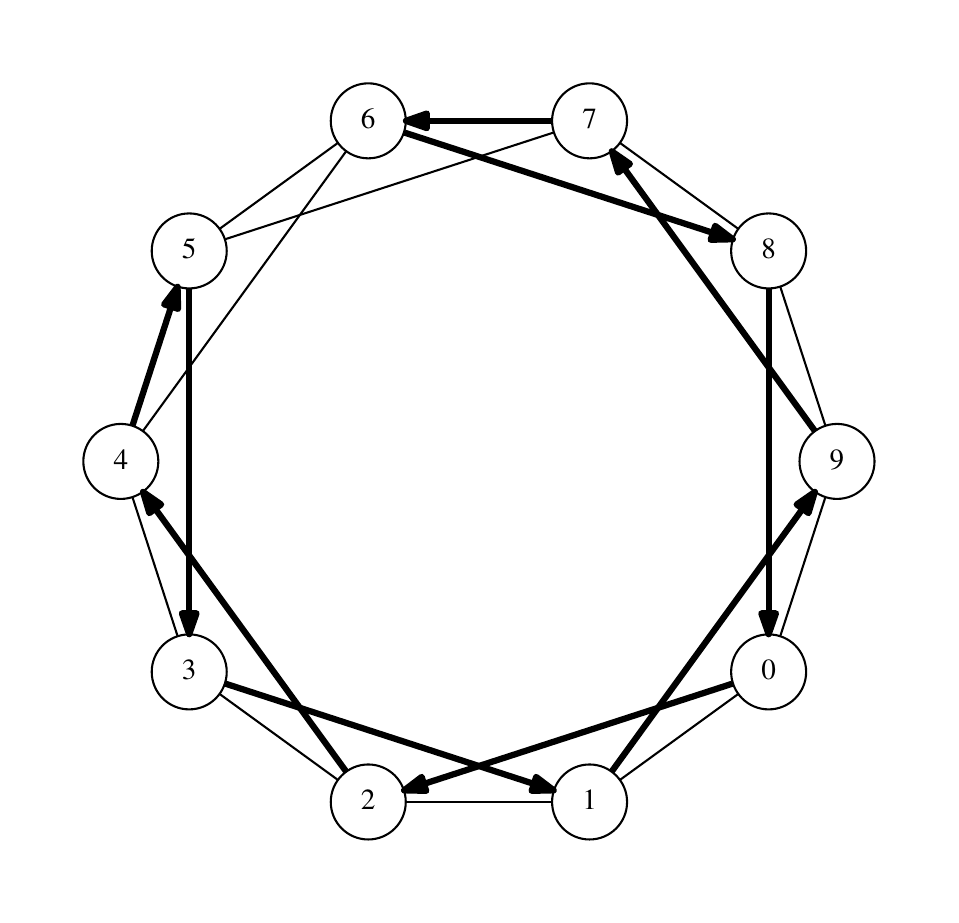}
\\
\textbf{(a)} & \textbf{(b)}
\end{tabular}
\caption{\textbf{(a)} The antiprism graph $C_{10}^{1,2}$. 
\textbf{(b)} A directed Hamiltonian cycle in $C_{10}^{1,2}$ that does not visit either of the edges $(4,6)$, $(5,6)$, $(5,7)$, i.e., has signature $000$ at node 4.}
\label{fig:anti}
\end{figure}

A \emph{cycle} is a closed walk without repeated edges, up to a choice of a starting/ending node.
A cycle is \emph{Hamiltonian} if it visits every node in the graph exactly once. 
A recurrence formula for the number of Hamiltonian cycles in $C_{2n}^{1,2}$ was first obtained by Golin and Leung \cite{golin}. Here we present a simpler derivation for the same formula.

For a subgraph $G$ of $C_{2n}^{1,2}$, we define the \emph{visitation signature} of $G$ at node $i$ as a triple of binary digits describing whether edges $(i,i+2)$, $(i+1,i+2)$, $(i+1,i+3)$ 
are visited by $G$, where digits $1/0$ mean visited/non-visited. Notice that these three edges form a path of length three in $C_{2n}^{1,2}$, as illustrated in Figure~\ref{fig:vis}. 
For example, a visitation signature $010$ means the second edge is in $G$, while the first and third edges are not.

A Hamiltonian cycle $Q$ (viewed as a subgraph) in $C_{2n}^{1,2}$ has one of the following two types:
\begin{itemize}
\item[(T1)]
There exists $i$ such that the visitation signature of $Q$ at $i$ is $000$;
\item[(T2)] 
For every $i$, the visitation signature of $Q$ at $i$ is not $000$.
\end{itemize}
Let us first enumerate Hamiltonian cycles of type (T1).
If a Hamiltonian cycle $Q$ has the visitation signature $000$ at $i$ (an example for $n=5$ and $i=4$ is given in Figure~\ref{fig:anti}b), 
then $Q$ must contain edges $(i+2,i+3)$ and $(i+2,i+4)$. It further follows that $Q$ must contain edges $(i+3,i+5)$ and $(i+4,i+6)$, 
and so on. Eventually, we conclude that $Q$ in this case is formed by two interweaving paths between nodes $i+2$ and $i+1$.
So, there exist exactly two directed Hamiltonian cycles (of opposite directions) that have visitation signature $000$ at $i$, and the value of $i$ is unique for such cycles.
In other words, there are two directed Hamiltonian cycles of type (T1) for each of the $2n$ nodes, totaling in $4n$ of such cycles.
Their generating function is 
\begin{equation}\label{eq:T1gen}
\sum_{n=3}^{\infty} 4n\cdot z^n = \frac{4z^3(3-2z)}{(1-z)^2}.
\end{equation}

To enumerate Hamiltonian cycles of type (T2), we need the following lemma:

\begin{figure}[!t]
\begin{tabular}{cccc}
\includegraphics[width=0.23\textwidth]{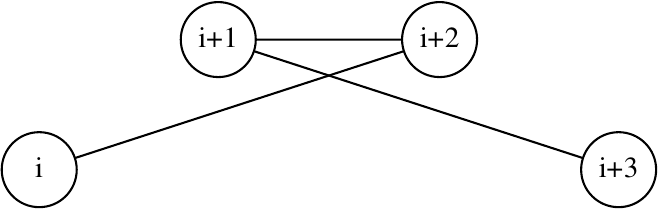} & \includegraphics[width=0.23\textwidth]{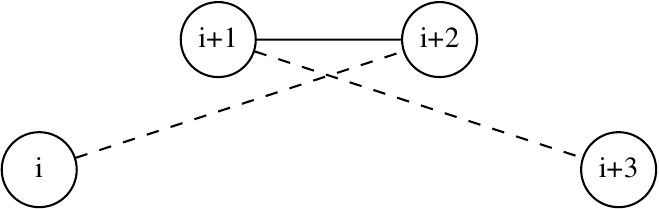} & \includegraphics[width=0.23\textwidth]{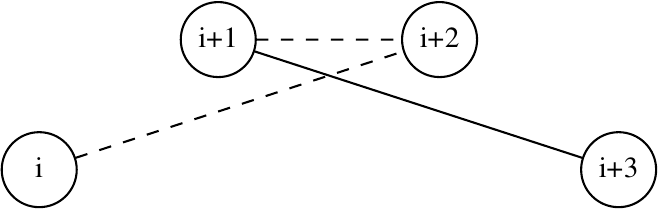} & \includegraphics[width=0.23\textwidth]{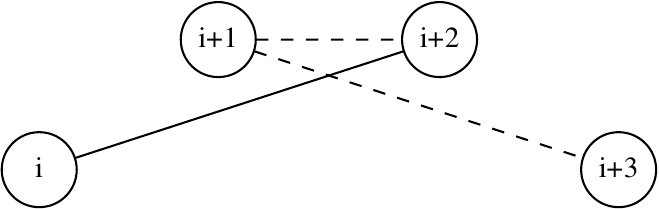} \\
$111$ & $010$ & $001$ & $100$ 
\end{tabular}
\caption{Possible visitation signatures for a Hamiltonian cycle of type (T2) in $C_{2n}^{1,2}$. Visited and non-visited edges are shown as solid and dashed, respectively.}
\label{fig:vis}
\end{figure}

\begin{lemma}\label{lem:no000} A subgraph $Q$ of $C_{2n}^{1,2}$ is a Hamiltonian cycle of type (T2) if and only if
(i) 
every node of $C_{2n}^{1,2}$ is incident to exactly two edges in $Q$; and 
(ii) the visitation signature of $Q$ at every node is $111$, $001$, $010$, or $100$ (shown in Figure~\ref{fig:vis}).
\end{lemma}

\begin{proof}
If $Q$ is a Hamiltonian cycle of type (T2), then condition (i) trivially holds.  We establish condition (ii), by showing 
that no other visitation signatures besides $111$, $001$, $010$, and $100$ are possible in $Q$.
Notice that: 
\begin{itemize}
\item The signature $000$ cannot happen anywhere in $Q$ by the definition of type (T2).
\item The signature $011$ at node $i$ implies the signature $000$ at node $i-1$.
\item The signature $110$ at node $i$ implies the signature $000$ at node $i+1$.
\item The signature $101$ at node $i$ implies the presence of edges $(i,i+1)$ and $(i+2,i+3)$ in $Q$; that is, $Q$ must coincide with the cycle $(i,i+2,i+3,i+1,i)$, a contradiction to $n\geq 3$.
\end{itemize}

Now, let $Q$ be a subgraph of $C_{2n}^{1,2}$ satisfying conditions (i) and (ii).
Let $Q' \subset Q$ be a connected component of $Q$.
Since every node is incident to two edges from $Q$, $Q'$ represents a cycle in $C_{2n}^{1,2}$. 

We claim that for any node $i$ of $C_{2n}^{1,2}$, $Q'$ contains either node $i+1$, or both of the nodes $i$ and $i+2$.
Indeed, if this statement does not hold for all nodes, then starting at a node belonging to $Q'$ and increasing its label by 1 or 2, keeping it in $Q'$, 
we can reach a node $k$ in $Q'$ such that neither $k+1$, nor $k+2$ are in $Q'$. 
Then $Q'$ (and $Q$) contains edges $(k-2,k)$ and $(k-1,k)$, and since every node in $Q$ is incident to exactly two edges, $Q$ does not contain edges $(k-1,k+1)$, $(k,k+1)$, and $(k,k+2)$. 
That is, the visitation signature of $Q$ at node $k-1$ is $000$, a contradiction to condition (ii), which proves our claim.

We say that $Q'$ \emph{skips} node $i$ if it contains nodes $i-1$ and $i+1$, but not $i$. 
If $Q'$ skips node $i$, consider a connected component $Q''$ of $Q$ that contains node $i$. By the aforementioned claim, the nodes of $Q'$ 
and $Q''$ must interweave, i.e., $Q' = (i-1, i+1, i+3, \dots)$ and $Q'' = (i, i+2, i+4, \dots)$. 
Then the signature of $Q$ at node $i$ is $101$, a contradiction to condition (ii), proving that $Q'$ cannot skip nodes. 
So, $Q'$ contains all the nodes of $C_{2n}^{1,2}$, and thus $Q'=Q$ represents a Hamiltonian cycle in $C_{2n}^{1,2}$.
\end{proof}

Lemma~\ref{lem:no000} allows us to obtain the number of Hamiltonian cycles of type (T2) in $C_{2n}^{1,2}$ as the number of subgraphs $Q$ satisfying conditions (i) and (ii).
To compute the number of such subgraphs, we construct a directed graph $S$ on the four allowed visitation signatures as nodes, where there is a directed edge $(s_1,s_2)$ 
whenever the signatures $s_1$ and $s_2$ can happen in $Q$ at two consecutive nodes. 
The graph $S$ and its adjacency matrix $B$ are shown in Figure~\ref{fig:sig}.

\begin{figure}[!t]
\begin{center}
\begin{tabular}{ccc}
\begin{tabular}{c}
\includegraphics[width=0.3\textwidth]{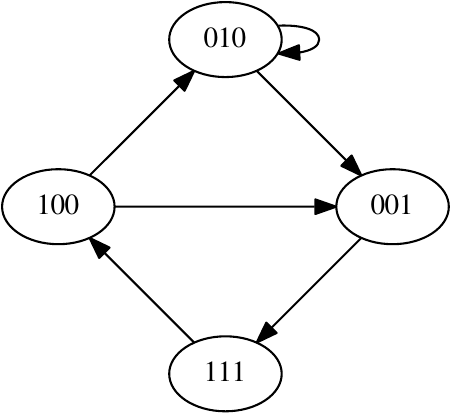}
\end{tabular}
&~\qquad~&
$
\begin{array}{c|cccc}
& 111 & 010 & 001 & 100 \\
\hline
111 & 0 & 0 & 0 & 1\\
010 & 0 & 1 & 1 & 0\\
001 & 1 & 0 & 0 & 0\\
100 & 0 & 1 & 1 & 0
\end{array}
$
\end{tabular}
\end{center}
\caption{The signature graph $S$ and its adjacency matrix $B$.}
\label{fig:sig}
\end{figure}

By Lemma~\ref{lem:no000} and Corollary~\ref{Th2}, the number of Hamiltonian cycles of type (T2) in $C_{2n}^{1,2}$ equals $\tr(B^{2n})$. 
Correspondingly, the total number of directed Hamiltonian cycles $h_n$ in $C_{2n}^{1,2}$ equals $4n + 2\tr(B^{2n})$; 
its generating function (derived from \eqref{eq:T1gen} and Theorem~\ref{th:trgen}) is
$$
\sum_{n=3}^{\infty} h_n\cdot z^n 
= \frac{4z^3(3-2z)}{(1-z)^2} + \frac{2z^3(10+11z+5z^2)}{1-z-2z^2-z^3}
= \frac{2z^3(16 - 19z -15z^2 + 3z^3 + 9z^4)}{(1-z)^2(1-z-2z^2-z^3)}.
$$
It further implies that the sequence $h_n$ satisfies the recurrence relation:
$$h_n = 3 h_{n-1} - h_{n-2} - 2h_{n-3} + h_{n-5},\qquad n\geq 8,$$
with the initial values $32, 58, 112, 220, 450, \dots$ for $n=3,4,\dots$ (sequence \texttt{A124353} in the OEIS~\cite{oeis}).

\section{Hamiltonian Cycles and Paths in Arbitrary Graphs}

Similarly to a Hamiltonian cycle, a \emph{Hamiltonian path} in a graph visits every node exactly once.  
Enumeration of Hamiltonian paths/cycles in an arbitrary graph represents a famous NP-complete problem. 
That is, one can hardly hope for the existence of an efficient (i.e., polynomial-time) algorithm for this enumeration 
and thus has to rely on less efficient algorithms of (sub)exponential time complexity. 
Below, we describe such a not-so-efficient, but very neat and simple algorithm,\footnote{We are not aware if this algorithm has been described in the literature before, 
but based on its simplicity we suspect that this may be the case.}
which is based on the transfer-matrix method and another basic combinatorial enumeration method called inclusion-exclusion (e.g., see \cite[Section~2.1]{stanley}).

We denote the number of (directed) Hamiltonian cycles and paths
in a graph $G$ by $\HC(G)$ and $\HP(G)$, 
respectively.

\begin{theorem} Let $G$ be a graph with node set $V$ and let $A$ be the adjacency matrix of $G$. 
Then
\begin{equation}\label{eq:HP}
\HP(G) = \sum_{S\subset V} (-1)^{|S|}\cdot \SUM\left(A_{V\setminus S}^{n-1}\right)
\end{equation}
and
\begin{equation}\label{eq:HC}
\HC(G) = \frac{1}{n} \sum_{S\subset V} (-1)^{|S|}\cdot \tr\left(A_{V\setminus S}^n\right),
\end{equation}
where $\SUM(M)$ denotes the sum of all\,\footnote{Alternatively, we can define $\SUM(M)$ as the sum of 
all \emph{non-diagonal} elements of $M$; formula \eqref{eq:HP} still holds in this case.} 
elements of a matrix $M$.
\end{theorem}

\begin{proof}
First, we notice that a Hamiltonian path in $G$ is the same as a walk of length $n-1$ that visits every node.
Indeed, a walk of length $n-1$ visits $n$ nodes, and if it visits every node in $G$, then it must visit each node only once. 
That is, such a walk is a Hamiltonian path.

For a subset $S\subset V$, we define $P_S$ as the set of all walks of length $n-1$ in $G$ that do not visit any node from $S$. 
Then by the principle of inclusion-exclusion, the number of Hamiltonian paths $\HP(G)$ is given by
$$\HP(G) = \sum_{S\subset V} (-1)^{|S|}\cdot |P_S|.$$
To use this formula for computing $\HP(G)$, it remains to evaluate $|P_S|$ for every $S\subset V$.  

Let $G_{V\setminus S}$ be the graph obtained from $G$ by removing all nodes (along with their incident edges) present in $S$, 
and let $A_{V\setminus S}$ be the adjacency matrix of $G_{V\setminus S}$. 
Then the elements of $P_S$ are nothing else but the walks of length $n-1$ in the graph $G_{V\setminus S}$. 
Hence, by Theorem~\ref{Th1}, $|P_S|$ equals $\SUM\left(A_{V\setminus S}^{n-1}\right)$,
which implies formula \eqref{eq:HP}.

Similarly, a Hamiltonian cycle in $G$ can be viewed as a closed walk of length $n$ that starts/ends at a node $v\in V$
and visits all nodes.
Hence, the number $\HC(G)$ of Hamiltonian cycles in $G$ can be computed by the formula
$$\HC(G) = \sum_{S\subset V\setminus\{v\}} (-1)^{|S|}\cdot \left(A_{V\setminus S}^n\right)_{v,v}.$$
Similar formulae hold if we view closed walks as starting/ending at a different node $v'\in V$. Averaging over the nodes in $V$ gives us formula \eqref{eq:HC}.
\end{proof}

Formulae \eqref{eq:HP} and \eqref{eq:HC} provide a practical method for 
computing $\HP(G)$ and $\HC(G)$, although they have exponential time complexity as
they sum $2^n$ terms (indexed by the subsets $S\subset V$). 
On a technical note, the matrix $A_{V\setminus S}$ can be obtained directly from the adjacency matrix $A$ of $G$ by removing the rows and columns corresponding to the nodes in $S$.

In an undirected graph $G$, the number of undirected Hamiltonian paths and cycles is given by $\frac{1}{2}\HP(G)$ and $\frac{1}{2}\HC(G)$, respectively.

\section{Simple Cycles and Paths of a Fixed Length}

Our approach for enumeration of Hamiltonian paths/cycles can be further extended to enumeration of \emph{simple} (i.e., visiting every node at most once) paths/cycles of a fixed length. 
We refer to simple paths and cycles of length $k$ as $k$-paths and $k$-cycles. 
We denote the number of (directed) $k$-cycles and $k$-paths in a graph $G$ by $\SC_k(G)$ and $\SP_k(G)$, respectively.

\begin{theorem}
Let $G$ be a graph with a node set $V=\{v_1,\dots,v_n\}$ and let $A$ be the adjacency matrix of $G$. 
Then, for an integer $k\geq 1$, 
\begin{equation}\label{eq:SC}
\SC_k(G) = \frac{1}{k} \sum_{T\subset V} \binom{n-|T|}{k-|T|}\cdot (-1)^{k-|T|}\cdot \tr\left(A_T^k\right)
\end{equation}
and
\begin{equation}\label{eq:SP}
\SP_k(G) = \sum_{T\subset V} \binom{n-|T|}{k+1-|T|}\cdot (-1)^{k+1-|T|}\cdot \SUM\left(A_T^k\right).
\end{equation}
\end{theorem}

\begin{proof}
If a $k$-cycle $c$ visits nodes from a set $U\subset V$, $|U|=k$, 
then $c$ represents a Hamiltonian cycle in the subgraph $G_U$ of $G$ induced by $U$. Hence, the number of $k$-cycles in $G$ equals
$$\SC_k(G) = \sum_{U\subset V,\ |U|=k} \HC(G_U).$$
By formula~\eqref{eq:HC}, we further have
\[
\begin{split}
\SC_k(G) &= \sum_{U\subset V,\ |U|=k}\frac{1}{k} \sum_{S\subset U} (-1)^{|S|}\cdot \tr\left(A_{U\setminus S}^k\right)\\
&= \frac{1}{k} \sum_{T\subset V}\ \ \sum_{U:\ T\subset U \subset V,\ |U|=k} (-1)^{k-|T|}\cdot \tr\left(A_T^k\right)\\
&= \frac{1}{k} \sum_{T\subset V} \binom{n-|T|}{k-|T|}\cdot (-1)^{k-|T|}\cdot \tr\left(A_T^k\right),
\end{split}
\]
which proves \eqref{eq:SC}. (Here $T$ stands for the set $U\setminus S$.)

If a $k$-path $p$ visits nodes from a set $U\subset V$, $|U|=k+1$, 
then $p$ represents a Hamiltonian path in the subgraph $G_U$ of $G$ induced by $U$.
Similarly to the above, we can employ formula \eqref{eq:HP} to obtain \eqref{eq:SP}.
\end{proof}

In an undirected graph $G$, the number of undirected $k$-cycles and $k$-paths is given by $\frac{1}{2}\SC_k(G)$ and $\frac{1}{2}\SP_k(G)$, respectively.

Using formula \eqref{eq:SC}, we have computed the number of $k$-cycles in the graph of the regular 24-cell for various values of $k$ 
(sequence \texttt{A167983} in the OEIS~\cite{oeis}).

\section*{Acknowledgments}
The work of the first author is supported by the National Science Foundation under grant no. IIS-1462107.

\end{document}